\documentclass[12pt,letterpaper]{amsart}
\usepackage{amsthm,amssymb,latexsym,amsmath}
\usepackage{graphicx,color,xcolor,comment}

\newtheorem{theorem}{Theorem}
\newtheorem{proposition}{Proposition}
\newtheorem{corollary}{Corollary}

\newtheorem{remark}{Remark}

\newtheorem{lemma}{Lemma}

\newcommand{\R}{{\mathbb{R}}}

\newcommand{\eps}{\varepsilon}

\newcommand{\udiv}{\mbox{div}\,}

\title[Vlasov-Fokker-Planck-Mean-Field]{A new approach to the mean-field limit of Vlasov-Fokker-Planck equations}

\author[D.~Bresch]{D. Bresch}
\address{\sc D. Bresch. Univ. Savoie Mont Blanc, CNRS, LAMA, 73000 Chamb\'ery, France}
\email{didier.bresch@univ-smb.fr}

\author[P.--E.~Jabin]{P.--E. Jabin}
\address{\sc P.--E. Jabin.  Department of Mathematics and Huck Institutes, Pennsylvania State University, State College, PA 16801, USA}
\email{pejabin@psu.edu}

\author[J.~Soler]{J. Soler}
\address{\sc J. Soler. Departamento de Matem\'atica Aplicada and Research Unit ``Modeling Nature'' (MNat), Facultad de Ciencias, Universidad de Granada, 18071 Granada, Spain}
\email{jsoler@ugr.es}
\begin{document}

\maketitle

\begin{abstract}
This article introduces a novel approach to the mean-field limit of stochastic systems of interacting particles, leading to the first ever derivation of the mean-field limit to the Vlasov-Poisson-Fokker-Planck system for plasmas in dimension~$2$ together with a partial result in dimension~$3$. The method is broadly compatible with second order systems that lead to kinetic equations and it relies on novel estimates on the BBGKY hierarchy. By taking advantage of the diffusion in velocity, those estimates bound weighted $L^p$ norms of the marginals or observables of the system, uniformly in the number of particles. This allows to qualitatively derive the mean-field limit for very singular interaction kernels between the particles, including repulsive Poisson interactions, together with quantitative estimates for a general kernel in  $L^2$.
\end{abstract}

\section{Introduction}
The rigorous derivation of kinetic models such as the Vlasov-Poisson system from many-particle systems has been a long standing open question, ever since the introduction of the Vlasov-Poisson system in~\cite{Vlasov1,Vlasov2}. While our understanding of the mean-field limit for singular interactions has made significant progress for first-order dynamics, the mean-field limit for second-order systems has remained frustratingly less understood. This article proposed a new approach that is broadly applicable to second-order systems with repulsive interactions and diffusion in velocity. In particular this allows to derive for the first time the Vlasov-Poisson-Fokker-Planck system  in dimension higher than one without any truncation or regularizing.

We more precisely consider the classical second-order Newton dynamics
\begin{equation}
\begin{split}
&\frac{d}{dt} X_i(t)=V_i(t),\quad X_i(t=0)=X_i^0,\\
&dV_i(t)=\frac{1}{N}\,\sum_{j\neq i} K(X_i-X_j)\,dt+\sigma\,dW_i,\quad V_i(t=0)=V_i^0,
\end{split}\label{Npart}
\end{equation}
where the $W_i$ are $N$~independent Wiener processes. For simplicity we take the positions $X_i$ on the torus~$\Pi^d$, while the velocities lie in $\R^d$. The kernel $K$ models the pairwise interaction between particles and is taken {\em repulsive} throughout this paper, in the basic sense that it derives from a potential $K=-\nabla \phi$ that is even and positive, $\phi\geq 0$.
\begin{remark}
For simplicity, we denote $\phi(0)=0$ and $K(0)=0$, even if $\phi$ and $K$ are not continuous at $0$. This notation simplifies the writing in the equation by allowing to sum over all $j$ in~\eqref{Npart} since the term $j=i$ trivially vanishes.
\end{remark}

We naturally focus on singular kernels~$K$ with, as a main guiding example, the case of Coulombian interactions
\begin{equation}
K=\alpha\,\frac{x}{|x|^d}+K_0(x),\label{coulombian}
\end{equation}
with $\alpha>0$ and $K_0$ a smooth correction to periodize $K$. This corresponds to the choice 
$\phi= \alpha\, (d-2)^{-1} |x|^{2-d} +\mbox{correction}$ if $d\geq 3$ or $\phi=-\alpha \ln |x|+\mbox{correction}$ if $d=2$.

The Coulombian kernel~\eqref{coulombian} typically models electrostatic interactions between point charges, such as ions or electrons in a plasma when the velocities are small enough with respect to the speed of light. In that setting,  diffusion in~\eqref{Npart} may for example represent collisions against a random background, such as the collision of the faster electrons against the background of ions. Such random collisions may also involve some friction in velocity, which we did not include in~\eqref{Npart} but could be added to our method without difficulty.  This makes~\eqref{Npart} with~\eqref{coulombian} one of the most classical and important starting points for the modeling of plasmas; we refer in particular to the classical~\cite{Bogo1,Bogo2}.

Coulombian interactions are also a natural scaling in many models. The obvious counterpart to plasmas concerns the Newtonian dynamics of point masses through gravitational interactions. This consists in taking $\alpha<0$ in \eqref{coulombian} and leads to attractive interactions with a negative potential and for this reason cannot be handled with the method presented here.

The system~\eqref{Npart} usually involves  a very large number of particles, typically up to $10^{20}-10^{25}$ in plasmas for example. This makes the mean-field limit especially attractive. This is a kinetic, Vlasov-Fokker-Planck equation posed on the limiting 1-particle density $f(t,x,v)$
\begin{equation} 
\partial_t f+v\cdot \nabla_x f+ (K\star_x \rho)\cdot\nabla_v f=\frac{\sigma^2}{2}\,\Delta_v f
\quad \hbox{ with } \quad \rho =  \int_{\R^d} f dv.
\label{vlasov}
\end{equation}
Well posedness for mean-field kinetic equations such as~\eqref{vlasov} is now reasonably well understood, including for singular Coulombian interactions such as~\eqref{coulombian} in dimension~$d\leq 3$. For the non-diffusive case $\sigma=0$, weak solutions were established in~\cite{Arsenev}, while classical solutions were obtained in dimension~$2$ in~\cite{UkaOka}. The dimension~$3$ case is harder and obtaining classical solutions requires more difficult dispersive arguments and were only obtained later in~\cite{LioPer,Pfa,Schaeffer}, see also the more recent~\cite{GasJabPer,HolMio,Loeper, Pal}. In the case with diffusion~$\sigma>0$, we refer to~\cite{Victory} for weak solutions, and to~\cite{Bouchut,De,OnoStr,ReiWec,VicOdw} for classical solutions.

Of course the mean-field scaling is not the only possible scaling on systems such as~\eqref{Npart}. One can in particular mention the likely even more critical Boltzmann-Grad limit, such as obtained in the classical~\cite{Lanford} and the recent major results in \cite{BodGalSaiSim1,BodGalSaiSim2,GalSaiTex,PulSafSim,PulSim}. We note as well that the derivation of macroscopic equations from mesoscopic systems such as~\eqref{vlasov} is another important and challenging questions. For example the passage to the fluid macroscopic system from Vlasov-Poisson-Fokker-Planck has been approached in different low-field (parabolic) or high-field (hyperbolic) regimes depending on the space dimension, see for example \cite{PS,NPS,GNPS,CChP} and references therein.

Mean-field limits have been rigorously derived for general systems, including second order dynamics such as~\eqref{Npart}, in the case of Lipschitz interaction kernels~$K$. We refer to the classical works~\cite{McKean,Sznitman} in the stochastic case and~\cite{BraHep,Dob} for the deterministic case. Uniform in time propagation of chaos has also been obtained in the locally Lipschitz case, notably in a close to convex case in~\cite{BoGuMa} and more recently in a non-convex setting in~\cite{GuLeMo}.

There now exists a large literature on the question of the mean-field limits, see for example the survey in~\cite{Gol,Jareview,JW2}.
However in the specific case of second order systems such as~\eqref{Npart}, very little is known. In dimension~$d=1$, the Vlasov-Poisson-Fokker-Planck system was derived in~\cite{GuLe,HaSa}. In dimension~$d\geq 2$, the only results for unbounded interaction kernels were obtained in~\cite{HauJab1,HauJab2}. But those are valid only in the deterministic case~$\sigma=0$, and for only mildly singulars kernels with $|K(x)|\lesssim |x|^{-\alpha}$, $|\nabla K|\lesssim |x|^{-\alpha-1}$ for $\alpha<1$. \cite{JabWan1} derived the mean-field limit with~$K\in L^\infty$ and without extra derivative. Those cannot cover Coulombian interactions, even in dimension~$2$. 

More is known for singular interaction kernels~$K$ that are smoothed or truncated at some N-dependent scale~$\eps_N$. In that truncated case, one can mention in particular~\cite{GanVic,GanVic2,VA,Wollman} for the convergence of so-called particle methods. The recent works~\cite{BoePic,Laz,LazPic} in the deterministic case and~\cite{HuLiPi} in the stochastic case considerably extended the results for such truncated kernels and allowed to almost reach the critical physical scale~$\eps_N\sim N^{-1/d}$. One can also mention~\cite{CaChSa} with polynomial cut-off. It is also possible to derive the Vlasov-Poisson system directly from many-particle quantum dynamics such as the Hartree equation, for which we briefly refer to~\cite{GolPau,La,Saf}.

The mean-field limits for first order systems with singular interactions appear to be more tractable. A classical example concerns the dynamics of point vortices or stochastic point vortices where the mean-field limit corresponds to the vorticity formulation of 2d incompressible Euler or Navier-Stokes. The interaction between vortices obey the Biot-Savart law which has the same singularity as the Coulombian kernel in dimension~$2$. In the deterministic case the mean-field limit was classically obtained for example in~\cite{Goodman91,GooHouLow90} or~\cite{Scho95,Scho96} for the 2d Euler and extended to remarkably essentially any Riesz kernels in~\cite{Se}. In the stochastic case, we refer in particular to~\cite{FHM-JEMS,Osada,JabWan2} for the limit to 2d Navier-Stokes, to~\cite{BrJaWa, BreJabWan} for singular attractive kernels, or to \cite{NgRoSe} for multiplicative noise. Uniform in time propagation of chaos was even recently obtained in~\cite{GuiLebMon,RoSe}. 

One of the reason second order systems appear more difficult to handle stems from how the structure of the singularity interacts with the distribution of velocities. Because of the term~$K(X_i-X_j)$, the singularity in pairwise interactions is typically localized on collisions~$X_i=X_j$. For first-order systems this corresponds to a point singularity, while for second-order systems the presence of the additional velocity variables makes it into a plane. In that regard, we also note that the derivation of macroscopic system directly from 2nd order dynamics is in fact better understood than the derivation of kinetic equations like~\eqref{vlasov}. We refer to the derivation of incompressible Euler in~\cite{HanIac}, or to the derivation of monokinetic solutions to~\eqref{vlasov} (which are essentially equivalent to a macroscopic system) in~\cite{Se}.

The main argument in our proof is a new quantitative estimate on the so-called marginals of the system through the BBGKY hierarchy. This leads to the propagation of some weighted $L^p$ estimates on the marginals. It implies a weak propagation of chaos in the sense of \cite{Sznitman} but it applies more broadly to initial data that are not chaotic or not close to being chaotic.

Recently new approaches have been introduced to bound marginals on systems with appropriate non-degenerate diffusion. Using relative entropy, \cite{Lacker} was the first to derive quantitative estimates comparing the marginals to the limiting tensorized solution, thus deriving optimal rates for the propagation of chaos in $O(1/N)$, instead of $O(1/\sqrt{N})$ on the convergence of the marginals (as observed for smoother interactions in \cite{Duer}). While formulated for 1st order method, the method also applies to 2nd order system with diffusion in velocity, as observed by the author. The method takes advantage of the regularizing provided by the diffusion to avoid ``losing'' a derivative in the hierarchy estimates. The use of the relative entropy however imposes that the interaction kernel belongs to an exponential Orlicz space. In a different context of non-exchangeable systems, \cite{JaPoSo} later used the propagation of $L^2$ norms on some equivalent of the marginals, again taking advantage of the diffusion but requiring the interaction kernel~$K$ in~$L^\infty$.

The present article focuses mostly on second-order singular systems, where our method combines this general idea with a specific choice of weights for the $L^p$ norms that are propagated. Those weights are based on a total energy reduced to $k$ particles when dealing with the marginal of order~$k$. They allow to take advantage of a further regularizing effect in the hierarchy to only require kernels~$K$ in some $L^p$ with $p>1$. The same idea to propagate $L^p$ norms on the marginals also applies to 1st order systems in confined domains, without then requiring weights.

A direct consequence of our approach is the first ever derivation of the mean-field limit for the repulsive Vlasov-Poisson-Fokker-Planck over a finite time interval. This applies to any chaotic initial data in dimension~$d=2$ and for initial data with more restrictive energy bound in any dimension~$d\geq 3$. We are expecting to extend this derivation in a future work to any chaotic initial data in any dimension~$d\geq 2$ by decomposing appropriately the initial data.

The paper is structured as follows: We start in Section 2 with the notations and main results. We first state our main result, Theorem~\ref{Conv},  that proves the convergence to the Vlasov-Fokker-Planck equation as $N$ tend to infinity  followed with a Theorem~\ref{quantitative} proving quantitative estimates for singular kernels in $L^2$. We next introduce Proposition \ref{propLp} which states the explicit propagation of weighted $L^p$ bounds on the marginals. We in particular discuss more thoroughly the limitations and possible extensions of our approach after stating Proposition~\ref{propLp}. Section 3 is devoted to the proof of Proposition~\ref{propLp} and Theorem~\ref{Conv} from the key technical contribution of the article around Lemma~\ref{technicalLemma}  and ends with the proof of Theorem~\ref{quantitative}.
%
\section{Main results}
\subsection{The new result}
We introduce 
the full $N$-particle joint law of the system $f_N$ which satisfies the Liouville or forward Kolmogorov equation
\begin{equation}
\begin{split}
& \partial_t f_N+\sum_{i=1}^N v_i\cdot\nabla_{x_i} f_N\\
&  \hskip1cm +\sum_{i=1}^N  \frac{1}{N}\,\sum_{j=1}^N K(x_i-x_j) \cdot\nabla_{v_i} f_N=\frac{\sigma^2}{2}\,\sum_i \Delta_{v_i} f_N, \label{Liouville}
\end{split}
\end{equation}
which is a linear advection-diffusion equation. However the marginals $f_{k,N}$ of $f_N$ will also play a critical role in the analysis. They correspond to the law of $k$ among $N$ particles and are represented through
\begin{equation}
\begin{split}
&f_{k,N}(t,x_1,v_1,\ldots,x_k,v_k) = \\
& \int_{\Pi^{d(N-k)}\times\R^{d(N-k)}} f_N(t,x_1,v_1\ldots,x_N,v_N)\,dx_{k+1}\,dv_{k+1}\dots dx_N\,dv_N.
\label{marginaldef}
\end{split}
\end{equation}
The question of well-posedness for Eq.~\eqref{Liouville} can be delicate and is separate from the issue of the mean-field limit that we consider here. For this reason, we consider a notion of entropy solution $f_N\in L^\infty(\R_+\times\Pi^{dN}\times\R^{dN})$ to~Eq.~\eqref{Liouville}, that is fully described later in subsection~\ref{Liouvillewellposed}, and to which we impose  some Gaussian decay in velocity
\begin{equation}
  \begin{split}
\sup_{t\leq 1}\int_{\Pi^{dN}\times \R^{dN}} e^{\beta\,\sum_{i\leq N} |v_i|^2}\,&f_N\,dx_1\,dv_1\dots dx_N\,dv_N\leq V^N,\\& \mbox{for some}\ \beta>0,\ V>0,
\end{split}
\label{gaussiandecay}
\end{equation}
for which we refer to the short discussion in subsection~\ref{Liouvillewellposed}.

Our main is the derivation of the mean-field limit for a broad class of singular kernels.
  \begin{theorem} \label{Conv} Assume that  that there exists some constant $\theta>0$ s.t. the potential $\phi$ satisfies
    \begin{equation}
    \int_{\Pi} e^{\theta\,\phi(x)}\,dx < +\infty,\label{expphi}
\end{equation}
and that
\[
K = - \nabla \phi \in L^p ({\Pi}^d) \qquad \hbox{for some}\qquad  p>1.
\]
Let $f$ be the unique smooth solution to the Vlasov equation~\eqref{vlasov} with initial data $f^0\in C^\infty(\Pi^d\times \R^d)$ such that $\int_{\Pi^d\times \R^d} f^0\,e^{\beta\,|v|^2}<\infty$. Consider moreover an entropy solution~$f_N$ to~\eqref{Liouville} in the sense of subsection~\ref{Liouvillewellposed} and satisfying~\eqref{gaussiandecay} with initial data $f_N^0\in L^\infty(\Pi^{dN}\times\R^{dN})$. Assume that  $f_{k,N}^0$ converges weakly in $L^1$ to $(f^0)^{\otimes k}$for each fixed $k$  and that  
\[
\|f_{k,N}^0\|_{L^\infty(\Pi^{dN}\times\R^{dN})}\leq M^k ,
\]
for some $M>0$ and for all $k\leq N$.
 Then there exists $T^*$ depending only on $M$, $V$, and $\|K\|_{L^p}$ such that $f_{k,N}$, given by \eqref{marginaldef},  weakly converge to $f_k= f^{\otimes k}$ in $L^q_{loc} ([0,\;T^\star] \times \Pi^{kd} \times \R^{kd})$ for any $k$, and any $2<q<\infty$, with $1/q+1/p\le 1$. 
 \end{theorem}
 
 Our estimates can also provide quantitative rates of convergence though we need to a stronger assumption, namely $K\in L^2$.
 \begin{theorem}  \label{quantitative}  
Assume the same conditions and hypotheses of {Theorem \ref{Conv}}, with moreover $p=2$. We also assume that there exists a constant $C$ independent of $N$ and $\varepsilon_N \to 0$
such that
$$\int_{\Pi^{kd}\times \R^{kd}}  |f_{k,N}^0 - (f^0)^{\otimes k}|^2 e^{\lambda(0) e_k}
 \le C^k \varepsilon_N,
$$
for all $k$ with 
\begin{equation}\label{ek}
e_k(x_1,v_1,\ldots,x_k,v_k)=\sum_{i\leq k} (1+|v_i|^2)+\frac{1}{N}\,\sum_{i,j\leq k} \phi(x_i-x_j)
\end{equation}
and
$$\lambda(t) =1/(\Lambda (1+t)) \hbox{ for a positive constant } \Lambda.$$
Then, there exists $T^*$ such that $f_{k,N}$ converges strongly to $f_k$ in $L^2_{loc} ([0,\;T^\star] \times \Pi^{kd} \times \R^{kd})$, for any $k$,  and we have the following quantitative estimate
$$
 \sup_{t\le T^\star} \int_{\Pi^{kd} \times \R^{kd}} |f_{N,k} - f^{\otimes k}|^2 e^{\lambda(t) e_k} 
\le {\widetilde C}^k \varepsilon_N ,
$$
for some $\widetilde C$ independent of $N$. 
\end{theorem}

In addition to the mean-field limit, Theorem \ref{Conv} implies the weak propagation of chaos in the sense of the famous~\cite{Sznitman}, although with strong conditions on $f_N^0$.
Theorem \ref{Conv} also justifies for the first time the convergence to the Vlasov-Poisson-Fokker-Planck in two space dimension. More precisely, we highlight the following  result

\begin{corollary} \label{ic} Let  $d=2$ and consider the Poisson kernel $K=-\nabla \phi$
with its associated potential $\phi(x)\simeq - \ln |x|$.  Then, the convergence properties given by Theorem \ref{Conv} 
hold true, leading to the Vlasov-Poisson-Fokker-Planck system.
\end{corollary} 
%
\subsection{New stability estimates}\label{stab}
Theorem~\ref{Conv} relies on a new approach to estimate on the BBGKY hierarchy solved by the marginals~$f_{k,N}$ which is of significant interest in itself. In general deriving bounds on either the BBGKY or limiting Vlasov hierarchy is complex. We refer for example to~\cite{GolMouRic} for the Vlasov hierarchy, to~\cite{DuerSain} for the study of long-time corrections to mean-field limits. Bounds on the hierarchy are critical for the derivation of collisional models such as the Boltzmann equation, ever since~\cite{Lanford}. Even a partial discussion of the challenges in the collisional setting would go well beyond the scope of this paper and we simply refer again to~\cite{BodGalSai,BodGalSaiSim1,BodGalSaiSim2,GalSaiTex,Kac1956,Lanford,PulSafSim,PulSim}.

The main difficulty in handling the hierarchy consists in the term 
\begin{equation}
\nabla_{v_i} \int_{\Pi^d\times\R^d} K(x_i-x_{k+1})\,f_{k+1,N}\,dx_{k+1}\,dv_{k+1},\label{extraterm}
\end{equation}
as seen in \eqref{hierarchy}, because this introduces the next order marginal~$f_{k+1,N}$ into the equation for $f_{k,N}$. When treated naively as a source term, it leads to a loss of one derivative on each equation of the hierarchy.

However it was recently noticed in first~\cite{Lacker} and then \cite{JaPoSo} that one may avoid this loss of derivative in the stochastic case for non-degenerate diffusion: Any $L^2$ estimate then gains an additional~$H^1$ dissipation which can be used to control the loss of one derivative. This idea still appears applicable in the present kinetic context: Even though we only have diffusion in velocity, the derivative in~\eqref{extraterm} is also only on the velocity variable.

Both~\cite{JaPoSo,Lacker} require high integrability on the kernel: $K\in L^\infty$ for~\cite{JaPoSo} and some sort of exponential Orlicz space of the type $\int e^{\lambda\,|K(x)|}\,dx<C$ for~\cite{Lacker}. Paper \cite{Lacker} used quantitative relative entropy estimates to prove uniqueness on the BBGKY hierarchy, while \cite{JaPoSo} proved uniqueness on a tree-indexed limiting hierarchy through~$L^2$ bounds. Hence in both case, the corresponding bounds on the marginals was already known uniformly in~$N$ and the challenge was to prove that the norm of the difference with the limit is small.

This leads to a first key difference with respect to the present approach and to the first critical new idea introduced in this paper. In essence, we note that the integral in~\eqref{extraterm} leads to a regularizing effect that has the same scaling as the convolution at the limit: One has by H\"older estimates that
\begin{equation}
  \begin{split}
    &\left\|\int_{\Pi^d} K(x_i-x_{k+1})\,f(x_1,\ldots,x_{k+1})\,dx_{k+1}\right\|_{L^q(\Pi^{dk})}
    \\
    &\qquad \qquad\qquad\qquad\qquad\qquad\qquad\leq \|K\|_{L^p(\Pi^d)}\,\|f\|_{L^q(\Pi^{d\,(k+1)})},
\end{split}
    \label{regularization}
  \end{equation}
provided that $1/p+1/q\leq 1$.

Taking advantage of~\eqref{regularization} for singular~$K\in L^p$ with $p$ small naturally leads to try to propagate~$L^q$ norms of the marginals~$f_{k,N}$ for large exponents~$q$; at the opposite of~\cite{JaPoSo,Lacker}. But it also leads to an additional major difficulty, due to the velocity variable in the unbounded space~$\R^d$ in~\eqref{extraterm}. In fact, trying to use~\eqref{regularization} in~\eqref{extraterm} as is would force the use of a mixed norm $L^q_x L^1_v$ on the marginals. Unfortunately such mixed norms are notoriously ill-behaved on kinetic equations.

Instead a more natural idea, from the point of view of kinetic equations, consists in using some moments or fast decay in velocity. Even if they are less usual for kinetic equations, the use of Gaussian moments is especially attractive in the current case because they are naturally tensorized. For example, one has the extension of~\eqref{regularization}
\begin{equation}
  \begin{split}
    &\int_{\Pi^{dk}\times\R^{dk}}e^{|v_1|^2+\ldots+|v_k|^2}\,\left|\int_{\Pi^d\times\R^d} K(x_i-x_{k+1})\,f_{k+1,N}\,dx_{k+1}\,dv_{k+1}\right|^q
    \\
    &\qquad\qquad \leq C_d\, \|K\|_{L^p(\Pi^d)}^q\,\int_{\Pi^{d(k+1)}\times\R^{d(k+1)}}e^{|v_1|^2+\ldots+|v_{k+1}|^2}\, | f_{k+1,N} |^q
\end{split}
    \label{regularization2}
\end{equation}
still provided $1/p+1/q\leq 1$.

However pure Gaussian moments in velocity does not seem to be naturally propagated at the discrete level of the hierarchy, even though they would trivially be propagated on the limiting mean-field equation at least for short time. This leads to the final critical idea of the paper, which is to incorporate the potential energy in the Gaussian: Namely to consider~$e^{\lambda(t)\,e_k}$ instead of a pure Gaussian with  $e_k$ defined by \eqref{ek}.

We observe that our definition of $e_k$ uses $1+|v_i|^2$ but could just as well be reduced to $|v_i|^2$ instead as \eqref{regularization2} suggests. The extra constant in $e_k$ allows to normalize the weight of each marginal by a factor $e^{\lambda(0) k}$ which saves some extra numerical constants in the proof.

We also remark that the use dynamical weights argument has been recently developed  in \cite{BreJabWan}
for first order particle systems with singular kernels. We also note that Proposition~\ref{propLp}, stated below, shows the propagation of weighted $L^q$ bounds on the marginals, without requiring the initial data to be chaotic or close to chaotic as introduced in \cite{Kac1956}. It hence applies to a broader framework than just the mean-field limit.
\begin{proposition}\label{propLp}
  Let us assume $K\in L^p(\Pi^d)$, for some $p>1$, and define
  \[
\lambda(t)=\frac{1}{\Lambda\,(1+t)}, \qquad
L  = \frac{C}{\lambda(1)^\theta} \|K\|_{L^p}^q ,
\]
for positive constants $\Lambda$, $C$, $\theta$ depending only on $q$, $d$ and $\sigma$ and $q$ and $1/q+1/p \le 1$.
Consider  a renormalized solution~$f_N$ to~\eqref{Liouville} satisfying~\eqref{gaussiandecay} with initial data $f_N^0\in L^\infty(\Pi^{dN}\times\R^{dN})$, satisfying 
\begin{equation}
\begin{split}
& \int_{\Pi^{kd}\times\R^{kd}} |f^0_{k,N}|^q\,e^{\lambda(0)\,e_k}\leq F_0^k, \\
&  \sup_{t\leq 1} \int_{\Pi^{Nd}\times\R^{Nd}} |f_{N}|^q\,e^{\lambda(t)\,e_N}\leq F^N,
\end{split} \label{assumptlq}
\end{equation}
for some $F>0$, $F_0>0$ and $q$ such that  $2\leq q<\infty$, with $1/q+1/p\leq 1$. Then, one has that
\begin{equation}\label{Estim}
\sup_{t\leq T}\int_{\Pi^{kd}\times\R^{kd}} |f_{k,N}|^q\,e^{\lambda(t)\,e_k}\leq 2^{k}\,F_0^k+F^k\,2^{2k-N-1},
\end{equation}
where $T$ is given by
\[
T = \min\left(1, \frac{1}{4\,L\,\max(F_0,F)}\right).
\]
  \end{proposition}

Proposition~\ref{propLp} shows that the corresponding $L^q$ norm of a marginal at order $k$ behaves like $C^k$ for some constant $C$. This is the expected scaling for propagation of chaos and tensorized marginals $f_k=f^{\otimes k}$.

However Proposition~\ref{propLp} also presents several intriguing features that we want to highlight.

\smallskip

$\bullet$ {\em Vlasov-Poisson-Fokker-Planck in higher dimensions.} Proposition~\ref{propLp} handles just as easily Coulombian interactions in any dimension $d$, and not only dimension~$d=2$ as Theorem~\ref{Conv}. Therefore,  Proposition~\ref{propLp} would imply some form of propagation of chaos for the Vlasov-Poisson-Fokker-Planck system in any dimension if we are able to consider initial $N$-particles laws $f_N^0$ which are
$f^0$-chaotic as $N \to +\infty$ and whose marginals $f^0$ and associated solution $f_n$ to the forward Kolmogorov equation satisfy \eqref{assumptlq}. While there are examples of such initial data, take $f_N^0=Z\,\exp(-e_N)$ for instance, they demand some sort of truncation or decay of the configurations with high energy. This is not satisfying because we cannot even take $f^0_N=(f^0)^{\otimes N}$: Assumption~\eqref{assumptlq} cannot hold in such a case as $e^{\lambda(0)\,e_k}$ is not integrable if $K$ is the Poisson kernel in dimension $d>2$. The issue is that by taking $f^0_N=(f^0)^{\otimes N}$, we allow some configurations with high potential energy. And roughly speaking the existence time $T$ in the proposition vanishes as the starting potential energy increases in that case. 

\smallskip

$\bullet$ {\em Repulsive potentials.} Proposition~\ref{propLp} does require repulsive potentials $\phi\geq 0$ as this assumption is critical in the proof. The repulsive assumption on the potential only appears to be needed to handle the discrete many-particle system. The extension to non-repulsive settings remains an open problem.

\smallskip

$\bullet$ {\em Extension to the stochastic case of mildly singular kernels}. A special case concerns mildly singular kernels $K$ with $K\in L^p$ for some $p>1$ s.t. $\phi\in L^\infty$. In that situation, by considering $\phi+\|\phi\|_{L^\infty}$ instead of $\phi$, yielding the same interaction kernel~$K$, we can always ensure that $\phi\geq 0$. For example this easily extends for the first time to the stochastic settings the results of~\cite{HauJab1,HauJab2}, that had been obtained only for deterministic second order systems with $|K|\lesssim |x|^{-\alpha}$ for $\alpha<1$.

\smallskip

$\bullet$ {\em Convergence for finite times.} We finally emphasize that Proposition~\ref{propLp}, just as Theorem~\ref{Conv},  holds over a finite time interval, independent of $N$. This may initially appear puzzling since we are dealing with linear equations for any fixed $N$. However because those estimates are essentially independent of $N$, they also extend to the non-linear limiting Vlasov equation. Moreover Proposition~\ref{propLp} includes a propagation of Gaussian moments in velocity over the marginals, from the term $e^{\lambda(t)\,e_k}$ and the definition~\eqref{ek} of $e_k$. The propagation for all times of such moments for Vlasov--Poisson is only known in dimension $d=2$, see~\cite{UkaOka,De}, and dimension~$d=3$,  see~\cite{Bouchut,GasJabPer,HolMio,LioPer,OnoStr,Pal,Pfa,ReiWec,Schaeffer,VicOdw} as cited in the introduction; it also requires in dimension~$3$ the use of dispersion estimates that are not present in our proof. As we already noted, Proposition~\ref{propLp} is in fact valid in any dimension which naturally limits it to some given finite time interval.

  \subsection{The case of first order system}
  While we focus on second order systems, we also emphasize that our method directly applies to first order systems on bounded domains (in a much simpler manner in fact) and provides the mean-field limit under very weak assumptions on the kernel~$K$ again.
  Consider in that case
\begin{equation}
\frac{d}{dt} X_i(t)=\frac{1}{N}\,\sum_{j\neq i} K(X_i-X_j)\,dt+\sigma\,dW_i,,\quad X_i(t=0)=X_i^0,
\label{Npart1st}
\end{equation}
fully on the torus~$\Pi^d$.  The mean-field limit is similar to~\eqref{vlasov}
\begin{equation} 
\partial_t f+ (K\star_x f)\cdot\nabla_x f=\frac{\sigma^2}{2}\,\Delta_x f.
\label{vlasov1st}
\end{equation}
Similarly the joint law $f_N(t,x_1,\ldots,x_N)$ solves an appropriately modified Liouville equation
\begin{equation}
 \partial_t f_N +\sum_{i=1}^N  \frac{1}{N}\,\sum_{j=1}^N K(x_i-x_j)\cdot
 \nabla_{x_i} f_N=\frac{\sigma^2}{2}\,\sum_i \Delta_{x_i} f_N. \label{Liouville1st}
\end{equation}
Because system~\eqref{Npart1st} does not involve velocities, many technical difficulties in our proofs actually vanish. For example, we do not need anymore to add assumptions such as~\eqref{gaussiandecay}. We do not need either to impose that $K$ derives from a potential and hence do not require assumptions like~\eqref{expphi} either.
  We then have the following equivalent of Theorem~\ref{Conv}.
  \begin{theorem} \label{Conv1st} Assume that  
\[
K \in L^p ({\Pi}^d) \quad \hbox{for some}\quad  p>1,\qquad (\udiv K)_-\in L^\infty(\Pi^d),
\]
where $x_-$ denotes the negative part of $x$.
Let $f$ be the unique smooth solution to the Vlasov equation~\eqref{vlasov1st} with initial data $f^0\in C^\infty(\Pi^d)$.  Consider moreover an entropy solution~$f_N$ to~\eqref{Liouville1st}, still in the sense of subsection~\ref{Liouvillewellposed},  with initial data $f_N^0\in L^\infty(\Pi^{dN})$. Assume that  $f_{k,N}^0$ converges weakly in $L^1$ to $(f^0)^{\otimes k}$for each fixed $k$  and that  
\[
\|f_{k,N}^0\|_{L^\infty(\Pi^{dN})}\leq M^k ,
\]
for some $M>0$ and for all $k\leq N$.
Then there exists $T^*$ depending only on $M$, $\|K\|_{L^p}$ and $\|(\udiv K)_-\|_{L^\infty}$ such that $f_{k,N}$, given by \eqref{marginaldef},  weakly converge to $f_k= f^{\otimes k}$ in $L^q_{loc} ([0,\;T^\star] \times \Pi^{kd})$ for any $k$, and any $2<q<\infty$, with $1/q+1/p\le 1$.
\end{theorem}
Because it is not our main focus, we do not give a distinct proof of Theorem~\ref{Conv1st}.

As mentioned above, there exists now a large literature for the mean field limit of 1st order systems in the stochastic case, with many recent progress for singular kernels. We refer for example to the derivation of 2d Navier-Stokes from a system of many vortices in \cite{Osada, FHM-JEMS, JabWan2}. The derivation of the 2d Keller-Segel system, corresponding to attractive Coulombian potentials, was recently obtained in \cite{BrJaWa,Tardy}, see also \cite{FouTar} for a precise description of the collisions leading to the blow-up. We also cite \cite{Lacker} which only requires the kernel to be in an Orlicz space similar to $Exp$, together with  \cite{LaLe} which obtains global in time regularity for Lipschiz kernels with a smallness assumption on ${\rm div} K$.

All those results require stronger assumptions on the kernel $K$ than just $K\in L^p$ with $p>1$ as here. A similar scaling was however obtained in the seminal \cite{Se}  on first order systems with no diffusion. The breakthrough method in \cite{Se} is based on a modulated energy between the empirical measure and the limit and it applies to Riesz kernels where $K \sim {1}/{|x|^\alpha}$ with $\alpha<d$ (corresponding to $K\in L^p$ with $p>1$), with either a repulsive gradient flow or Hamiltonian interactions, or alternatively where $K \ast f \in W^{1, \infty}$. Uniform in time propagation of chaos was later obtained in \cite{RoSe} including diffusion  with the restriction $\alpha<d-1$ using the modulated energy method and some relaxation rates properties. This was recently improved in \cite{CRS} to again $\alpha<d$ combining precised relaxation rates with the new modulated free energy introduced in \cite{BrJaWa}.  One obvious advantage of our method here is that it allows for a much more general form of interaction, with singularities far away from the origin. On the other hand, Theorem~\ref{Conv1st} does require a non-vanishing diffusion and is again only valid for a finite time, instead of the much stronger uniform in time estimates above.

Contrary to the case of 2nd order systems, this short time limitation appears less fundamental as many limiting systems do not blow up, with the obvious exception of attractive interactions such as Keller-Segel. We conjecture that the present method could lead to large time results by taking advantage of the full non-degenerate diffusion for 1st order systems.

%
\subsection{Our notion of entropy solution for the hierarchy; the well-posedness of Eq.~\eqref{Liouville}\label{Liouvillewellposed}}
\subsubsection{The definition}
Being non-linear, our estimates cannot be performed on any weak solutions. Moreover, the concept of solution for $f_N$ are carried over the marginals $f_{k,N}$ and not just the joint law $f_N$ so that we also need an appropriate notion of entropy solutions on those marginals.

\smallskip

\noindent {\it The hierarchy for the marginals from the Liouville equation.} 
From Eq.~\eqref{Liouville}, the $f_{k,N}$ solve the so-called BBGKY hierarchy
\begin{equation}
\begin{split}
  \partial_t f_{k,N}&+\sum_{i=1}^k v_i\cdot\nabla_{x_i} f_{k,N}+\sum_{i\leq k}\,\frac{1}{N}\,\sum_{j\leq k} K(x_i-x_j)\cdot\nabla_{v_i} f_{k,N}\\
  & +\frac{N-k}{N} \sum_{i\leq k} \nabla_{v_i} \cdot \int_{\Pi^d\times \R^d} f_{k+1,N}\,K(x_i-x_{k+1}) dx_{k+1} dv_{k+1}\\
  &=\frac{\sigma^2}{2}\,\sum_{i\leq k} \Delta_{v_i} f_{k,N}.
  \end{split}\label{hierarchy}
\end{equation}
If $f_N$ belongs to $L^\infty$ and verifies~\eqref{gaussiandecay}, then all marginals $f_{k,N}$ belong to $L^\infty_t L^q_{x,v}$ for every $q<\infty$ with similar Gaussian decay. For simplicity, we denote here abstractly $L^q_{x,v}$ any space $L^q(\Pi^{kd}\times\R^{kd})$ when there is no confusion about the dimension $k$,  as in our case. We also denote by $L^q_{\lambda e_k}$ the weighted $L^q$ space
\[
\|f\|_{L^q_{\lambda e_k}}^q=\int_{\Pi^{kd}\times\R^{kd}} |f|^q\,e^{\lambda\,e_k}.
\]
Since $K\in L^p$, for some $p>1$, then by using a direct H\"older inequality those bounds on the $f_{k,N}$ implies that  
\[
\int_{\Pi^d\times \R^d} f_{k+1,N}\,K(x_i-x_{k+1})\,dx_{k+1}\,dv_{k+1}\in L^\infty_t L^q_{x,v},
\]
for all $q<\infty$. This allows to immediately and rigorously derive~\eqref{hierarchy} from Eq.~\eqref{Liouville}. 

\medskip

\noindent {\it  Definition of entropy solutions.} 
We denote the advection component of~\eqref{hierarchy}
\begin{equation} \label{L}
L_k=\sum_{i\leq k} v_i\cdot\nabla_{x_i}+\frac{1}{N}\,\sum_{i,j\leq k} K(x_i-x_j)\cdot\nabla_{v_i}.
\end{equation}
The argument above implies that the only difficulties to propagate our estimates in~\eqref{hierarchy} stem from $L_k$. Consequently we define our entropy solution in the following manner: A function~$f_N\in L^\infty([0,\ 1]\times\Pi^{dN}\times \R^{dN})$ satisfying~\eqref{gaussiandecay} is an entropy solution iff all marginals~$f_{k,N}$ for $1\leq k\leq N$, as defined by~\eqref{marginaldef}, satisfy that for any $T\in [0,\ 1]$, any $1<q<\infty$ and any $\lambda<\lambda_0$
\begin{equation}
\begin{split}
\int_0^T\!\!\int_{\Pi^{dk}\times\R^{dk}} &e^{\lambda\,e_k}\,|f_{k,N}|^{q-1}\,\\&
\mbox{sign}(f_{k,N}) \,L_k\,f_{k,N}\,dx_1\,dv_1\dots dx_k\,dv_k\,dt\geq 0.
\end{split}
\label{entropyformal}
\end{equation}
Inequality~\eqref{entropyformal} is still somewhat formal and should be understood in the following rigorous sense: For some smooth convolution kernel~$K_\eps$, one has that
\begin{equation}
\begin{split}
  \liminf_{\eps\to 0}\int_0^T\!\!\int_{\Pi^{dk}\times\R^{dk}}&e^{\lambda\,e_k}\, |K_\eps^{\otimes k}\star f_{k,N}|^{q-1}\, \mbox{sign}(K_\eps^{\otimes k}\star f_{k,N})\\
  &K_\eps^{\otimes k}\star\,(L_k\,f_{k,N})\,dx_1\,dv_1\dots dx_k\,dv_k\,dt\geq 0, \label{entropyrigorous}
\end{split}
\end{equation}
where we denote
\[
\begin{split}
  K_\eps^{\otimes k}\star\,g=\int_{\Pi^{dk}\times\R^{dk}} &K_\eps(x_1-y_1,v_1-w_1)\,\dots K_\eps(x_k-y_k,v_k-w_k)\\
  &g(y_1,w_1,\ldots,y_k,w_k)\,dy_1\,dw_1\dots dy_k\,dw_k
\end{split}
\]
with $K_\eps \to \delta$ when $\eps \to 0$.
However it is usually more delicate to determine whether any weak solution~$f_N$ in $L^\infty$ and with the bound~\eqref{gaussiandecay} is an entropy solution according to our definition. For linear advection-diffusion equations such as~\eqref{Liouville}, this is usually approached through the notion of renormalized solutions as introduced in~\cite{DL}. In that context,~\eqref{entropyrigorous} is obviously similar to the classical commutator estimate at the basis of many methods for renormalized solutions.

\begin{remark}
{\rm 1)} We first remark that~\eqref{entropyformal} is automatically satisfied if we have classical  solutions. Indeed $L_k$ is an antisymmetric operator so that we expect it to propagate $L^q$ norms so that if all terms are smooth
\[
|f_{k,N}|^{q-1}\, \mbox{sign}(f_{k,N}) \,L_k f_{k,N}=L_k\,|f_{k,N}|^q.
\]

\noindent {\rm 2)} We immediately observe that the reduced energy $e_k$ is formally invariant under the advection component of~\eqref{hierarchy}:
\[
L_k\,e_k=\frac{2}{N}\,\sum_{i,j\leq k} v_i\cdot \nabla_{x_i} \phi(x_i-x_j)+\frac{2}{N}\,\sum_{i,j\leq k} K(x_i-x_j)\cdot v_i=0,
\]
since $K=-\nabla_x \phi$. In the same way, we have $L_k\,\Phi(e_k)=0$, for any locally Lipschitz function $\Phi$. 

\noindent {\rm 3)} If $K$ is smooth and $f_N$ is a classical solution to~\eqref{Liouville}, we would hence immediately have equality in \eqref{entropyformal}. With $K$ only in $L^p$, it would be straightforward to obtain one entropy solution in the sense defined above, through passing to the limit in a sequence of solutions for a smoother kernel~$K$.
\end{remark}

\begin{remark} 
There exists an extensive literature on renormalized solutions with a comparably large variety of potential assumptions that one may consider. While we cannot do justice to this question in this short discussion, we briefly mention for instance~\cite{Ha} that studies the specific case of the Liouville equation~\eqref{Liouville} for second order systems without diffusion. In the present setting of a constant non-vanishing diffusion, we also refer to~\cite{BoKrRoSh,LeLi,LebLio} that provide broad results of well-posedness for velocity fields in $L^p$. 

We in particular note that renormalized solutions apply to the case $K\in L^p$ with $p> 2$ and $f_N$ in $L^\infty$  with $\nabla_{v_i} f_N^{q/2}\in L^2$ for any $q<\infty$ and satisfying the extension of~\eqref{gaussiandecay},
\[
\sup_{t\leq 1}\int_{\Pi^{dN}\times\R^{dN}} e^{\lambda_0\,e_k}\,f_N\,dx_1\,dv_1\dots dx_N\,dv_N<\infty.
\]
The latter estimates are natural for the Liouville~\eqref{Liouville}, as demonstrated by Lemma~\ref{technicalLemma} for the case $k=N$ in Section~\ref{proof}.
In that situation, all marginals $f_{k,N}$ belong to $L^\infty_t L^q_{x,v}$ for every $q<\infty$ with similar exponential decay in $e_k$ and with as well $\nabla_{v_i} f_{k,N}^{q/2}\in L^r_{t,x,v}$ for any $r<2$.  
This regularity easily allows to prove that~\eqref{entropyrigorous} holds for  $\lambda<\lambda_0$.

We also mention that so-called mild solutions can also offer a natural way to prove~\eqref{entropyrigorous}. We simply refer to~\cite{Bouchut,CS} for such formulations through the Fokker--Planck kernel in whole space, or to~\cite{Clark} or~\cite{De,VicOdw} for periodic conditions.
\end{remark}
\subsubsection{\it Strong solutions up to the first collision.}  We also emphasize that, in the case of repulsive kernels smooth out of the origin but with singular potentials $\lim_{x\to 0} \phi(x)=+\infty$, a straightforward bound on the energy of the system can easily lead  to strong solutions on the many-particle system~\eqref{Npart}, bypassing the need for entropy or renormalized solutions.

Very roughly, if $K\in C^\infty(\Pi^d\setminus\{0\})$, then up to the conditional time of first collision in~\eqref{Npart}, we may write that
\[
d\left(\sum_{i=1}^N |V_i|^2+\frac{1}{N}\,\sum_{i\neq j} \phi(X_i-X_j)\right)=\sigma^2\,dt
+\sum_{i=1}^N 2\,\sigma \, V_i\cdot dW_i.
\]
This implies that, with probability $1$, the total energy remains finite if it was so initially. Because $\lim_{x\to 0} \phi(x)=+\infty$, it also implies that collisions almost surely never happen. This argument would in particular apply to the Coulombian case in any dimension~$d\geq 2$.

To conclude this discussion of the well posedness of~\eqref{Liouville} or~\eqref{Npart} for a fixed~$N$, we emphasize the estimates that we described here cannot easily be made uniform in $N$. The previous discussion of the energy bound on the system~\eqref{Npart} for the Coulombian interaction in dimension~$d=2$ is an excellent illustration: If we have the following bound
\[
\sum_{i=1}^N |V_i|^2+\frac{1}{N}\,\sum_{i\neq j} \phi(X_i-X_j)\leq E
\]
with some large probability on some time interval, and for $\phi(x)=-\log |x|$ then this only proves that for any $i\neq j$
\[
|X_i-X_j|\geq e^{-N\,E},
\]
which is indeed finite for any fixed~$N$ but is completely unhelpful when considering the limit $N\to \infty$.

Hence the present discussion remains focused on renormalized solutions for a fixed~$N$. Quantitative approach to renormalized solutions have for example been introduced in~\cite{CrDe}, which are based on the propagation of a sort of $\log$-derivative on the characteristics; see also for example the discussion on Eulerian variants in~\cite{BrJa}. This leads to an interesting and so far mostly fully open question as to whether it would be possible to obtain quantitative bounds that would combine the limit $N\to\infty$ with some regularity estimates on the solution for a fixed~$N$.  
%
\section{Proof of the main results\label{proof}}

\subsection{The BBGKY and Vlasov hierarchies\label{hierarchies}}
Using~\eqref{vlasov}, the tensorized limits $f_k= {\overline f}^{\otimes k}$ satisfy the following Vlasov hierarchy
 \begin{equation}
\begin{split}
 \partial_t f_k &+ \sum_{i=1}^k v_i\cdot\nabla_{x_i} f_k\\
&  +\sum_{i=1}^k  (K\star  \int_{\R^d}  f dv) \cdot\nabla_{v_i} f _k 
=\frac{\sigma^2}{2}\,\sum_{i=1}^k \Delta_{v_i} f_k. 
\end{split}\label{vlasovhierarchy}
\end{equation}
To avoid repeating the analysis working on \eqref{hierarchy} or \eqref{vlasovhierarchy}, we introduce the generalized hierarchy equation
\begin{equation}
\begin{split}
 & \hspace{-0.5cm} \partial_t F_{k,N}+\sum_{i=1}^k v_i\cdot\nabla_{x_i} F_{k,N}+\sum_{i\leq k}\,\frac{\gamma}{N}\,\sum_{j\leq k} K(x_i-x_j)\cdot\nabla_{v_i} F_{k,N}\\
  & +\frac{N-\gamma k}{N}\sum_{i\leq k} \nabla_{v_i} \cdot \int_{\Pi^d\times \R^d} F_{k+1,N}\,K(x_i-x_{k+1})dx_{k+1}dv_{k+1}\\
  &=\frac{\sigma^2}{2}\,\sum_{i\leq k} \Delta_{v_i} F_{k,N}+R_{k,N}.
  \end{split}\label{hierarchyext}
\end{equation}
Note that Eq.~\eqref{hierarchyext} is exactly Eq.~\eqref{vlasovhierarchy} for $\gamma=0,\; R_{k,N}=0$ and exactly Eq.~\eqref{hierarchy} for $\gamma=1,\; R_{k,N}=0$. In the same spirit we denote
\[\begin{split}
&e_{k,\gamma}=\sum_{i\leq k} (1+|v_i|^2)+\frac{\gamma}{N}\,\sum_{i,j\leq k} \phi(x_i-x_j),\\
&L_{k,\gamma}=\sum_{i\leq k} v_i\cdot\nabla_{x_i}+\frac{\gamma}{N}\,\sum_{i,j\leq k} K(x_i-x_j)\cdot\nabla_{v_i}
\end{split}
\]
and observe that we of course still have $L_{k,\gamma}\,e_{k,\gamma}=0$.

The main technical contribution of this section and of the paper is Lemma~\ref{technicalLemma} stated in subsection~\ref{sectechnicalLemma}, which provides estimates for the solutions to \eqref{hierarchy}.
We will then use the uniform bound on the $k$-marginals $f_{k,N}$ for the proof of Prop.~\ref{propLp}. Prop~\ref{propLp} allows passing to the limit in the hierarchy~\eqref{hierarchy} and a final use of Lemma~\ref{technicalLemma} leads to prove uniqueness to the limiting hierarchy~\eqref{vlasovhierarchy} to conclude result of Theorem~\ref{Conv}.
%
\subsection{The key technical lemma}\label{sectechnicalLemma}
We first present the key technical lemma
which links the $k$-marginal $L^q_w$ control to the $(k+1)$-marginal $L^q_w$ estimate control.  
\begin{lemma} \label{technicalLemma}
  Assume that $K\in L^p(\Pi^d)$, for some $p>1$. There exist some constants $\Lambda$, $C$, $\theta$ depending only on $q$, $d$ and $\sigma$ s.t
  \[
\begin{split}
  \|F_{k,N}\|_{L^q_{\lambda(t)\,e_k}}^q \leq &\|F_{k,N}(t=0)\|_{L^q_{\lambda(0) \,e_k}}^q\\& +q\,\int_0^t\int |F_{k,N}|^{q-1}\,\mbox{sign}(F_{k,N}) R_{k,N}\,e^{\lambda(s)\,e_{k,\gamma}}\,ds \\
  &+ k\,\frac{N-\gamma\,k}{N}\,\frac{C}{\lambda^{\theta}(t)}\,\|K\|^q_{L^p}\,\int_0^t \|F_{k+1,N}(s)\|_{L^q_{\lambda(s)\,e_{k+1}}}^q\,ds, 
\end{split}
\]
for any entropy solution $F_{k,N}$ to~\eqref{hierarchyext} in the sense of subsection~\ref{Liouvillewellposed} and satisfying~\eqref{gaussiandecay} with $F_{k,N}\in L^q_{\lambda(t)\,e_{k,\gamma}}$, and for any $2\leq q<\infty$ such that $1/q+1/p\leq 1$, with $\lambda(t)$ defined by  $\lambda(t)=\frac{1}{\Lambda\,(1+t)}$.
  \end{lemma}
\begin{proof}
To be made fully rigorous, many calculations in this proof should involve a convolution kernel $K_\eps$, estimating
\[
\frac{d}{dt}\int |K_\eps^{\otimes k}\star F_{k,N}|^{q}\,e^{\lambda(t)\,e_{k,\gamma}},
\]
and passing to the limit in $\eps\to 0$ while using appropriately the entropy condition~\eqref{entropyrigorous}. For simplicity however, we will only present the corresponding formal calculations.

We hence calculate in a straightforward manner
\[
\begin{split}
  &\frac{d}{dt}\int |F_{k,N}|^{q}\,e^{\lambda(t)\,e_{k,\gamma}}=q\,\int |F_{k,N}|^{q-1}\,
  \mbox{sign}(F_{k,N}) \partial_t F_{k,N}\,e^{\lambda(t)\,e_{k,\gamma}}\\
  &\quad+\lambda'(t)\,\int e_{k,\gamma}\,|F_{k,N}|^{q}\,e^{\lambda(t)\,e_{k,\gamma}}.
  \end{split}
\]
Inserting now in this identity the definition of $\lambda(t)$ and the  equation~\eqref{hierarchy} we find
\[
\begin{split}
  &\frac{d}{dt}\int |F_{k,N}|^{q}\,e^{\lambda(t)\,e_{k,\gamma}}=-q\,\int |F_{k,N}|^{q-1}\,
  \mbox{sign}(F_{k,N}) (L_{k,\gamma} \,F_{k,N})\,e^{\lambda(t)\,e_{k,\gamma}}\\
  &\quad +q\,\frac{\sigma^2}{2}\,\int |F_{k,N}|^{q-1}\,\mbox{sign}(F_{k,N}) \left(\sum_{i\leq k} \Delta_{v_i}\,F_{k,N}\right)\,e^{\lambda(t)\,e_{k,\gamma}}\\
  &\quad-q\,\frac{N-\gamma\,k}{N}\,\sum_{i\leq k}\int |F_{k,N}|^{q-1}\,\mbox{sign}(F_{k,N}) \\
  &\qquad\qquad\qquad\nabla_{v_i}\cdot\int K(x_i-x_{k+1})\,F_{k+1,N}\,dx_{k+1}\,dv_{k+1}\,e^{\lambda(t)\,e_{k,\gamma}}\\
  &\quad-\Lambda\,\lambda^2(t)\,\int e_{k,\gamma}\,|F_{k,N}|^{q}\,e^{\lambda(t)\,e_{k,\gamma}} +q\,\int |F_{k,N}|^{q-1}\,\mbox{sign}(F_{k,N}) \, R_{k,N}\,e^{\lambda(t)\,e_{k,\gamma}}.
  \end{split}
\]
Note that
\[
q\,|F_{k,N}|^{q-1}\,\mbox{sign}(F_{k,N})\,(L_{k,\gamma}\,F_{k,N})=L_{k,\gamma} \,|F_{k,N}|^q,
\]
so that by integration by parts, we formally have that
\[
\begin{split}
  &q\,\int |F_{k,N}|^{q-1}\,\mbox{sign}(F_{k,N})\,(L_{k,\gamma}\,F_{k,N})\, e^{\lambda(t)\,e_{k,\gamma}}\\
& \qquad  = -\int |F_{k,N}|^q\,L_{k,\gamma}\,e^{\lambda(t)\,e_{k,\gamma}}=0.
\end{split}
\]
On the other hand, again by integration by parts
\[
\begin{split}
  &q\,\frac{\sigma^2}{2}\,\int |F_{k,N}|^{q-1}\,\mbox{sign}(F_{k,N}) \,\left(\sum_{i\leq k} \Delta_{v_i}\,F_{k,N}\right)\,e^{\lambda(t)\,e_{k,\gamma}}\\
  &\quad=-q\,(q-1)\,\sum_{i\leq k}\,\frac{\sigma^2}{2}\,\int |F_{k,N}|^{q-2}\,|\nabla_{v_i} F_{k,N}|^2\,e^{\lambda(t)\,e_{k,\gamma}}\\
  &\quad-2\,q\,\lambda(t)\,\sum_{i\leq k}\,\frac{\sigma^2}{2}\,\int |F_{k,N}|^{q-1}\,
  \mbox{sign}(F_{k,N})\,v_i\cdot\nabla_{v_i} F_{k,N}\,e^{\lambda(t)\,e_{k,\gamma}}.
\end{split}
\]
By Cauchy-Schwartz, since $q\ge 2$, we obtain that
\[
\begin{split}
&q\,\frac{\sigma^2}{2}\,\int |F_{k,N}|^{q-1}\,\mbox{sign}(F_{k,N})\,\left(\sum_{i\leq k} \Delta_{v_i}\,F_{k,N}\right)\,e^{\lambda(t)\,e_{k,\gamma}}\\
  &\quad\leq -q\,(q-1)\,\sum_{i\leq k}\,\frac{\sigma^2}{4}\,\int |F_{k,N}|^{q-2}\,|\nabla_{v_i} F_{k,N}|^2\,e^{\lambda(t)\,e_{k,\gamma}}\\
  &\qquad+  \frac{q}{q-1}\, \lambda^2\,\frac{\sigma^2}{2} \,\int |F_{k,N}|^q\,\sum_{i\leq k} |v_i|^2\,e^{\lambda(t)\,e_{k,\gamma}}. 
\end{split}
\]
Note that since $\phi\geq 0$, we have that $\sum_{i\leq k} |v_i|^2\leq e_k$ and, therefore, combining all our estimates so far, we deduce that
\[
\begin{split}
  &\frac{d}{dt}\int |F_{k,N}|^{q}\,e^{\lambda(t)\,e_{k,\gamma}} \leq -q(q-1)\sum_{i\leq k}\,\frac{\sigma^2}{4}\,\int |F_{k,N}|^{q-2}\,|\nabla_{v_i} F_{k,N}|^2\,e^{\lambda(t)\,e_{k,\gamma}}\\
  &\ -q\,\frac{N-\gamma\,k}{N}\,\sum_{i\leq k}\int |F_{k,N}|^{q-1}\,\mbox{sign}(F_{k,N})\\
  &\qquad\qquad\qquad\nabla_{v_i}\cdot\int K(x_i-x_{k+1})\,F_{k+1,N}\,dx_{k+1}\,dv_{k+1}\,e^{\lambda(t)\,e_{k,\gamma}}\\
  &\ -\frac{\Lambda}{2}\,\lambda^2(t)\,\int e_{k,\gamma}\,|F_{k,N}|^{q}\,e^{\lambda(t)\,e_{k,\gamma}} +q\,\int |F_{k,N}|^{q-1}\,\mbox{sign}(F_{k,N})\,R_{k,N}\,e^{\lambda(t)\,e_{k,\gamma}},
  \end{split}
\]
provided that $\Lambda\geq  \frac{q}{q-1}\,\sigma^2$.

We integrate by parts the second term in the right-hand side to obtain
\[
\begin{split}
  &\sum_{i\leq k}\int |F_{k,N}|^{q-1}\,\mbox{sign}(F_{k,N}) \\
  &\qquad\nabla_{v_i}\cdot\int K(x_i-x_{k+1})\,F_{k+1,N}\,dx_{k+1}\,dv_{k+1}\,e^{\lambda(t)\,e_{k,\gamma}}=RH_1+RH_2,\\
\end{split}
\]
with
\[
\begin{split}
  RH_1 = - (q-1)\, \sum_{i\leq k}&\int |F_{k,N}|^{q-2}\nabla_{v_i} F_{k,N}\\
  &\times \int K(x_i-x_{k+1})\,F_{k+1,N}\,dx_{k+1}\,dv_{k+1}\,e^{\lambda(t)\,e_{k,\gamma}},\\
\end{split}
\]
and
\[
\begin{split}
  RH_2= - 2\lambda(t)\, \sum_{i\leq k}&\int |F_{k,N}|^{q-1}\,
  \mbox{sign}(F_{k,N})\,v_i\\
  & \times \int K(x_i-x_{k+1})\,F_{k+1,N}\,dx_{k+1}\,dv_{k+1}\,e^{\lambda(t)\,e_{k,\gamma}}.
\end{split}
\]
We perform a straightforward Cauchy-Schwartz inequality on both terms to find that
\[
\begin{split}
  RH_2
 & \leq \lambda^2(t) \,\sum_{i\leq k} \int |F_{k,N}|^q\,|v_i|^2\,e^{\lambda(t)\,e_{k,\gamma}}\\
  &\quad + \sum_{i\leq k} \int  |F_{k,N}|^{q-2} \left|\int K(x_i-x_{k+1}) F_{k+1,N}\,dx_{k+1}\,dv_{k+1}\right|^2\,e^{\lambda(t)e_{k,\gamma}},
\end{split}
\]
and similarly
\[
\begin{split}
  &RH_1\leq \frac{\sigma^2}{4}\,\sum_{i\leq k} \int |F_{k,N}|^{q-2}\,|\nabla_{v_i} F_{k,N}|^2\,e^{\lambda(t)\,e_{k,\gamma}}\\
  &+\frac{(q-1)^2}{\sigma^2}\sum_{i\leq k} \int |F_{k,N}|^{q-2}\left|\int K(x_i\!-\!x_{k+1})F_{k+1,N}dx_{k+1}dv_{k+1}\right|^2e^{\lambda(t)e_{k,\gamma}}.
\end{split}
\]
Note that by  Young estimates
\[
\begin{split}
  & \int |F_{k,N}|^{q-2}\,\left|\int K(x_i-x_{k+1})\,\,F_{k+1,N}\,dx_{k+1}\,dv_{k+1}\right|^2\,e^{\lambda(t)\,e_{k,\gamma}} \\
  &\qquad\leq \frac{q-2}{q}\,\lambda^2\,\int |F_{k,N}|^q\,e^{\lambda(t)\,e_{k,\gamma}}\\
&\qquad   +\frac{2}{q\,\lambda^{q-2}}\,\int e^{\lambda(t)\,e_{k,\gamma}}\,\left|\int K(x_i-x_{k+1})\,\,F_{k+1,N}\,dx_{k+1}\,dv_{k+1}\right|^q. 
\end{split}
\]
Therefore, combining together all those terms, we obtain the further estimate
\[
\begin{split}
  \sum_{i\leq k}\int &|F_{k,N}|^{q-1} \mbox{sign}\left(F_{k,N}\right)\nabla_{v_i}\cdot\int K(x_i\!-\!x_{k+1})\,F_{k+1,N}dx_{k+1}dv_{k+1}e^{\lambda(t)\,e_{k,\gamma}}\\
  &\leq \frac{\sigma^2}{4}\,\sum_{i\leq k} \int |F_{k,N}|^{q-2}\,|\nabla_{v_i} F_{k,N}|^2\,e^{\lambda(t)\,e_{k,\gamma}}\\
  &\quad+ \lambda^2(t)\,\left( 1+\frac{(q-2)\,(q-1)^2}{q\,\sigma^2}\right)\, \sum_{i\leq k} \int |F_{k,N}|^q\,(1+|v_i|^2)\,e^{\lambda(t)\,e_{k,\gamma}}\\
  &\quad +\frac{2}{q\,\lambda^{q-2}}\left( 1+\frac{(q-1)^2}{\sigma^2}\right)\\
  &\hspace{2cm}\times \sum_{i\leq k}\int e^{\lambda(t)\,e_{k,\gamma}}\left|\int K(x_i-x_{k+1})\,F_{k+1,N}\,dx_{k+1} dv_{k+1}\right|^q.
\end{split}
\]
Hence, provided that
\[
 \Lambda\geq 2q\,\left(1+\,\frac{(q-2) (q-1)^2}{q\,\sigma^2}\right),
\]
we obtain that
\[
\begin{split}
  &\frac{d}{dt}\int |F_{k,N}|^{q}\,e^{\lambda(t)\,e_{k,\gamma}}   \\
  &\le C_{q,\sigma,d}\,k\frac{N\!-\!\gamma k}{\lambda^{q-2}\,N}\,\int e^{\lambda(t)\,e_{k,\gamma}}\left|\int K(x_1\!-\!x_{k+1})F_{k+1,N}dx_{k+1}dv_{k+1}\right|^q. \\
\end{split}
\]
At this point  is where we take advantage of the specific structure of the hierarchy. 
Denoting $q^*$ the conjugate of $q$ namely such that $1/q^*+1/q=1$, we bound
\[
\begin{split}
 &\hspace{-1cm} \left|\int  K(x_1-x_{k+1})\,F_{k+1,N}\,dx_{k+1}\,dv_{k+1}\right|^q\\
  & \leq \left(\int |K(x_1-x_{k+1})|^{q^*} \,e^{-\frac{q^*}{q}\,\lambda(t)\,|v_{k+1}|^2}\,dx_{k+1}\,dv_{k+1}\right)^{q/q^*}\\
  &\qquad \times \int |F_{k+1,N}|^q\,e^{\lambda(t)\,|v_{k+1}|^2}\,dx_{k+1}\,dv_{k+1},\\
\end{split}
\]
which implies
\[
\begin{split}
  &\hspace{-1cm} \left|\int K(x_1-x_{k+1})\,F_{k+1,N}\,dx_{k+1}\,dv_{k+1}\right|^q\\
& \leq \frac{C_{q,\sigma,d}}{\lambda^{q\,d/2q^*}(t)}\,\|K\|^q_{L^p}\,\int |F_{k+1,N}|^q\,e^{\lambda(t)\,|v_{k+1}|^2}\,dx_{k+1}\,dv_{k+1},
\end{split}
\]
since $q\geq p^*$. Consequently
\[
\begin{split}
  &\int e^{\lambda(t)\,e_{k,\gamma}}\,\left|\int K(x_1-x_{k+1})\,\,F_{k+1,N}\,dx_{k+1}\,dv_{k+1}\right|^q\\
& \leq   \frac{C_{q,\sigma,d}}{\lambda^{qd/2q^*}(t)}\|K\|^q_{L^p}\int |F_{k+1,N}|^q\,e^{\lambda(t)\,|v_{k+1}|^2+\lambda(t)\,e_{k,\gamma}}\,dx_1dv_1\dots dx_{k+1}dv_{k+1}.
  \end{split}
\]
Note that
\[
e_{k+1,\gamma}=e_{k,\gamma}+1+|v_{k+1}|^2+\frac{2\,\gamma}{N}\,\sum_{i\leq k} \phi(x_i-x_{k+1})\geq e_{k,\gamma}+1+|v_{k+1}|^2,
\]
so that
\[
\begin{split}
  &\int e^{\lambda(t)\,e_k} \left|\int K(x_i-x_{k+1})\,\,f_{k+1,N}\,dx_{k+1}\,dv_{k+1}\right|^q\\
&\quad \leq   \frac{C_{q,\sigma,d}}{\lambda^{qd/2q^*}(t)}\,\|K\|^q_{L^p}\,\int |f_{k+1,N}|^q\,e^{\lambda(t)\,e_{k+1}}\,\,dx_1\,dv_1\dots dx_{k+1}\,dv_{k+1}.
  \end{split}
\]
This finally lets us conclude, as claimed, that
\[
\begin{split}
  \frac{d}{dt}\int |f_{k,N}|^{q}\,e^{\lambda(t)\,e_{k,\gamma}} & \leq k\,\frac{N-\gamma\,k}{N}\,\frac{C_{q,\sigma},d}{\lambda^{\theta_{q,d}}(t)}\,\|K\|^q_{L^p}\,\int |f_{k+1,N}|^q\,e^{\lambda(t)\,e_{k+1,\gamma}}\\
  &\quad +q\,\int |F_{k,N}|^{q-1}\,\mbox{sign}(F_{k,N})\,R_{k,N}\,e^{\lambda(t)\,e_{k,\gamma}}.
\end{split}
\]
\end{proof}
%
\subsection{Proof of technical results} \label{conclusion}
We start this subsection with the proof of the Proposition \ref{propLp}.
\begin{proof}[Proof of Proposition~\ref{propLp}]
From the analysis in subsection~\ref{hierarchies} and the assumptions \eqref{gaussiandecay} and~\eqref{assumptlq} of Proposition~\ref{propLp}, we have that $F_{k,N} = f_{k,N}$ is a renormalized solution to~\eqref{hierarchy} and thus~\eqref{hierarchyext} with $\gamma=1$. Moreover, $f_{k,N}$ satisfies the other assumptions  in Lemma \ref{technicalLemma} with $R_{k,N}=0$. Denoting
\[
X_k(t)=\int |f_{k,N}|^{q}\,e^{\lambda(t)\,e_k},
\]
we hence observe that, by Lemma~\ref{technicalLemma}, we have the coupled dynamical inequality system
\[
X_k(t)\leq X_k(0)+k\,L\,\int_0^t X_{k+1}(s)\,ds,
\]
for any $t\in [0,\ 1]$,  where
\[
L=\frac{C}{\lambda^\theta(1)}\,\|K\|^q_{L^p}.
\]
From the assumptions of Proposition~\ref{propLp}, we immediately have that
\begin{equation}
X_k(t)\leq F_0^k+k\,L\,\int_0^t X_{k+1}(s)\,ds.\label{basicinduction}
\end{equation}
We now invoke the following simple lemma
\begin{lemma}\label{indu} 
  Consider any sequence $X_k(t)$ satisfying~\eqref{basicinduction} then one has that
  \begin{equation}
  \begin{split}
X_k(t)\leq &\sum_{l=k}^m F_0^l\,L^{l-k}\,t^{l-k}\,\frac{(l-1)!}{(k-1)!\,(l-k)!}\\
&+L^{m+1-k}\,\int_0^t X_{m+1}(s)\,(t-s)^{m-k}\,\frac{m!}{(k-1)!\,(m-k)!}\,ds.
\end{split}\label{induction}
  \end{equation}
  \label{lemsimple}
  \end{lemma}
Assuming that Lemma~\ref{lemsimple} holds, we use \eqref{induction} up to $m+1=N$ to derive through the assumptions on $f_N$ that
\[
\begin{split}
X_k(t)\leq &\sum_{l=k}^{N-1} F_0^l\,L^{l-k}\,t^{l-k}\,\frac{(l-1)!}{(k-1)!\,(l-k)!}\\
& \hskip1cm +L^{N-k}\,\int_0^t F^N\,(t-s)^{N-1-k}\,\frac{(N-1)!}{(k-1)!\,(N-1-k)!}\,ds,
\end{split}
\]
that is
\begin{equation}
\begin{split}
&X_k(t)\leq \sum_{l=k}^{N-1} F_0^l\,L^{l-k}\,t^{l-k}\,\frac{(l-1)!}{(k-1)!\,(l-k)!} \\
&  \hskip1.5cm +F^N\,L^{N-k}\,t^{N-k}\,\frac{(N-1)!}{(k-1)!\,(N-k)!}.\label{finalinduction}
\end{split} 
\end{equation}
Note that
\[
\frac{(l-1)!}{(k-1)!\,(l-k)!}=\binom{l-1}{k-1}\leq 2^{l-1}.
\]
Hence, \eqref{finalinduction} implies that
\[
\begin{split}
  X_k(t)&\leq \sum_{l=k}^{N} F_0^l\,L^{l-k}\,t^{l-k}\,2^{l-1}+F^N\,L^{N-k}\,t^{N-k}\,2^{N-1}\\
  &=2^{k-1}\, F^k_0\,\sum_{l=k}^{N-1} F_0^{l-k}\,2^{l-k}\,L^{l-k}\,t^{l-k}
      +F^k 2^{k-1}  F^{N-k} \,L^{N-k}\,t^{N-k}\,2^{N-k}\\
  &\leq  2^{k-1}\, F_0^k (2- 2^{k-N+1}) + F^k 2^{k-1} 2^{k-N}\\
  & \leq F_0^k 2^k  + F^k 2^{2k-N-1}
\end{split}
\]
provided that $4\,L\,t\,\max(F_0,F)<1$, which concludes the proof of the proposition. \end{proof}

We finish with the quick proof of  Lemma~\ref{lemsimple}.
\begin{proof}[Proof of Lemma~\ref{lemsimple}]
Taking $m=k$ in \eqref{induction}, we get
\[
X_k(t)\leq  F_0^k+L\,\int_0^t X_{k+1}(s)\,\frac{k!}{(k-1)!\,(k-k)!}\,ds,
\]
which is our starting point. Moreover assuming that~\eqref{induction} holds for $m$, we may use~\eqref{basicinduction} to find
\[
\begin{split}
&X_k(t)\leq \sum_{l=k}^m F_0^l\,L^{l-k}\,t^{l-k}\,\frac{(l-1)!}{(k-1)!\,(l-k)!}\\
  &+L^{m+1-k}\,\int_0^t \left(F_0^{m+1}+L\,(m+1)\,\int_0^s X_{m+2}(s)\,ds\right)\\
  &\qquad\qquad\qquad\qquad\qquad(t-s)^{m-k}\,\frac{m!}{(k-1)!\,(m-k)!}\,ds.
\end{split}
\]
This yields
\[
\begin{split}
  X_k(t)\leq  &\sum_{l=k}^m F_0^l\,L^{l-k}\,t^{l-k}\,\frac{(l-1)!}{(k-1)!\,(l-k)!}\\
  &+L^{m+1-k}\,F_0^{m+1}\,\frac{m!}{(k-1)!\,(m-k)!}\,\int_0^t (t-s)^{m-k}\,ds\\
&+L^{m+2-k} \int_0^t X_{m+2}(r) \int_r^t (t-s)^{m-k}\,ds\,dr \frac{(m+1)!}{(k-1)!\,(m-k)!}\,ds,
\end{split}
\]
or
\[
\begin{split}
  X_k(t)&\leq \sum_{l=k}^m F_0^l\,L^{l-k}\,t^{l-k}\,\frac{(l-1)!}{(k-1)!\,(l-k)!}\\
  &+L^{m+1-k}\,F_0^{m+1}\,\frac{m!}{(k-1)!\,(m+1-k)!}\, t^{m+1-k}\\
&+L^{m+2-k}\int_0^t X_{m+2}(r) (t-r)^{m+1-k}\,ds\,dr \frac{(m+1)!}{(k-1)!\,(m+1-k)!}\,ds,
\end{split}
\]
as claimed.
\end{proof}

\subsection{Proof of Theorem \ref{Conv}}
The proof of Theorem~\ref{Conv} follows closely the steps in the proof of Proposition~\ref{propLp}, once appropriate bounds have been derived. 
\medskip

\noindent 1){ \em Uniform bounds on $f_N$ in $L^q_{e_N}$}. First of all, note that from the assumptions of Theorem~\ref{Conv}, we can easily obtain a bound on $f_N^0$ in $L^q_{\lambda^0\,e_N}$ for $\Lambda$ large enough. Indeed
\[
\begin{split}
  &\int_{\Pi^{dN}\times \R^{dN}} |f_N^0|^q\,e^{\lambda^0\,e_N}\\
  &\qquad=e^N\,\int_{\Pi^{dN}\times \R^{dN}} |f_N^0|^q\,e^{2\,\lambda^0\,\sum_{i\leq N} |v_i|^2}\,e^{\frac{\lambda^0}{N}\,\sum_{i,j\leq N} \phi(x_i-x_j)-\lambda^0\,\sum_{i\leq N} |v_i|^2}.
\end{split}
\]
We have straightforward $L^r$ estimates on $e^{\frac{\lambda^0}{N}\,\sum_{i,j\leq N} \phi(x_i-x_j)-\lambda^0\,\sum_{i\leq N} |v_i|^2}$ as by H\"older inequality,
\[
\begin{split}
  \int_{\Pi^{dN}\times \R^{dN}} & e^{\frac{r\,\lambda^0}{N}\,\sum_{i,j\leq N} \phi(x_i-x_j)-r\,\lambda^0\,\sum_{i\leq N} |v_i|^2}\\
  &=\frac{C^N}{\lambda_0^{N/2}}\,\int_{\Pi^{dN}} e^{\frac{r\,\lambda^0}{N}\,\sum_{i,j\leq N} \phi(x_i-x_j)}\\
  & \leq \frac{C^N}{\lambda_0^{N/2}}\,\left(\Pi_{i\leq N} \int_{\Pi^{dN}} e^{r\,\lambda^0\sum_{j\leq N} \phi(x_i-x_j)}\right)^{1/N}\leq \frac{C^N}{\lambda_0^{N/2}},
\end{split}
\]
from some constant $C$ and by assumption~\eqref{expphi} in Theorem~\ref{Conv}, provided that $r\,\lambda^0\leq 1/\theta$. This implies, again by H\"older inequality
\[
\begin{split}
  \int_{\Pi^{dN}\times \R^{dN}} & |f_N^0|^q\,e^{\lambda^0\,e_N}\leq \frac{C^N}{\lambda_0^{N/2}}\,\int_{\Pi^{dN}\times \R^{dN}} |f_N^0|^{r^*\,q}\,e^{2\,r^*\,\lambda^0\,\sum_{i\leq N} |v_i|^2}\\
  &\leq \frac{C^N}{\lambda_0^{N/2}}\,\|f_N^0\|_{L^\infty}^{q\,r^*-1}\,\int_{\Pi^{dN}\times \R^{dN}} |f_N^0|^{r^*\,q}\,e^{2\,r^*\,\lambda^0\,\sum_{i\leq N} |v_i|^2}.
\end{split}
\]
Using now assumption~\eqref{gaussiandecay}, provided that $2\,r^*\,\lambda_0\leq\beta$, we conclude that 
\begin{equation}
\int_{\Pi^{dN}\times \R^{dN}} |f_N^0|^q\,e^{\lambda^0\,e_N}\leq \left(\frac{C\,V\,M}{\lambda_0}\right)^N,\label{boundfN0}
\end{equation}
for any $q<\infty$. We now choose any fixed $2<q<\infty$ such that $1/p+1/q<1$ and we remark that the Liouville Eq.~\eqref{Liouville} is included in Eq~\eqref{hierarchyext} for $\gamma=1$, $R_{k,N}=0$, and $k=N$. Thus, we next invoke Lemma~\ref{technicalLemma} for $f_N$ with $k=N$ and $\gamma=1$, to find that $f_N$ solves
\[
\frac{d}{dt} \int_{\Pi^{dN}\times \R^{dN}} |f_N(t,.,.)|^q\,e^{\lambda(t)\,e_N}\leq 0,
\]
so that from~\eqref{boundfN0}, we obtain that
\[
\sup_{t\leq 1}\int_{\Pi^{dN}\times \R^{dN}} |f_N(t,.,.)|^q\,e^{\lambda(t)\,e_N}\leq \left(\frac{C\,V\,M}{\lambda_0}\right)^N.
\]
This finally implies that there exists some constant $F>0$ such that
\begin{equation}\label{fNest}
\sup_{t\leq 1}\int_{\Pi^{dN}\times \R^{dN}} |f_N(t,.,.)|^q\,e^{\lambda(t)\,e_N}\leq F^N.
  \end{equation}
 
 \medskip

 \noindent 2) {\em Uniform estimates on the marginals and passing the limit in the hierarchy~\eqref{hierarchy}.} First of all we can perform the same bounds on each $f_{k,N}^0$ to find similarly to~\eqref{boundfN0} that
 \[
 \int_{\Pi^{kd}\times \R^{kd}} |f_{k,N}^0|^q\,e^{\lambda^0\,e_k}\leq \left(\frac{C\,V\,M}{\lambda_0}\right)^k.
 \]
 As a consequence,  every assumptions of Proposition~\ref{propLp} hold, and, in particular, assumption~\eqref{assumptlq}. This implies that for some time $T^*>0$, depending only on $V$, $M$, $\|K\|_{L^p}$ and the choice of $q$, we have that 
 \[
\sup_N\sup_{t\leq T^*}\int_{\Pi^{kd}\times \R^{kd}} |f_{k,N}|^q\,e^{\lambda(t)\,e_k}\leq \bar M^k,
 \]
 for some constant $\bar M$. At this point, we will not need anymore the potential in the reduced energy $e_k$, which was required to handle the $L_k$ operator that vanishes at the limit. For this reason, and since $\phi\geq 0$, we deduce from the previous inequality
 \begin{equation}
\sup_N\sup_{t\leq T^*}\int_{\Pi^{kd}\times \R^{kd}} |f_{k,N}|^q\,e^{\lambda(T^*)\,\sum_{i\leq k} |v_i|^2}\leq \bar M^k.\label{estfkN}
\end{equation}
These uniform bounds let us extract a converging subsequence such that all $f_{k,N}$ converge weak-$\star$ to some $\bar f_k$ in $L^\infty([0,\ T^*],\;L^q_{x,v})$ that also satisfies 
 \begin{equation}
\sup_{t\leq T^*}\int_{\Pi^{kd}\times \R^{kd}} |\bar f_{k}|^q\,e^{\lambda(T^*)\,\sum_{i\leq k} |v_i|^2}\leq \bar M^k,\label{estbarfk}
\end{equation}
where we have used classical convex estimates. We emphasize that for the moment we only have convergence of a subsequence, though we still denote it by $N$ for simplicity. We eventually obtain the convergence of the whole sequence only after the uniqueness of the limit is proved in the next step.

 From estimate~\eqref{estfkN}, and since $1/q+1/p\leq 1$, we may simply bound
 \[
\left\|\sum_{i\leq k}\frac{1}{N}\sum_{j\leq k} K(x_i-x_j)\cdot\nabla_{v_i} f_{k,N}\right\|_{L^\infty_t L^1_{x,v,loc}}\lesssim \frac{k^2}{N}\,\|K\|_{L^p}\,\|f_{k,N}\|_{L^\infty_t L^q_{x,v}}.
 \]
 For any fixed $k$, the corresponding term vanishes as $N\to\infty$. Similarly estimate~\eqref{estfkN} allows to pass to the limit
 \[
 \begin{split}
   &\int_{\Pi^d\times\R^d} K(x_i-x_{k+1})\,f_{k+1,N}\,dx_{k+1}\,dv_{k+1}\\
   &\qquad\longrightarrow \int_{\Pi^d\times\R^d} K(x_i-x_{k+1})\,\bar f_{k+1}\,dx_{k+1}\,dv_{k+1},
\end{split}
   \]
   for the weak-$\star$ topology of $L^\infty([0,\ T^*],\;L^q_{x,v})$. It is straightforward to pass to the limit in the sense of distributions in all other terms of the hierarchy~\eqref{hierarchy} so that we deduce that $\bar f_k$ is a solution to the limiting hierarchy~\eqref{vlasovhierarchy} in the sense of distributions.

   We can also easily identify the initial value of $\bar f_k$. From~\eqref{hierarchy} and the bounds derived from~\eqref{estfkN}, we immediately obtain a uniform bound on $\partial_t f_{k,N}$ in $L^\infty_t W^{-1,q}_{x,v,loc}$. By the assumption of Theorem~\ref{Conv}, $f_{k,N}^0$ converges weakly to $(f^0)^{\otimes k}$ so that we have $\bar f_k(t=0)=(f^0)^{\otimes k}$.
 
\medskip 
 
\noindent  3) {\em Uniqueness on the limiting hierarchy and conclusion}. We first argue that $\bar f_k$ is automatically a renormalized solution to~\eqref{vlasovhierarchy}. Indeed Eq.~\eqref{vlasovhierarchy} can be seen as a linear advection-diffusion equation with a locally Lipschitz velocity field $(v_1,\ldots,v_k)$ and a remainder
\[
\nabla_{v_i}\cdot \int_{\Pi^d\times\R^d} K(x_i-x_j)\,\bar f_{k+1}\,dx_{k+1}\,dv_{k+1},
\]
that belongs to $L^\infty_t L^q_{x,v}$ with $q>2$ per our prior estimates.

Next we note that since $f$ is a classical solution to  the vlasov equation \eqref{vlasov}, the $f^{\otimes k}$  also yield renormalized solutions to the    Vlasov hierarchy \eqref{vlasovhierarchy} for every $k\geq 1$.  Due to the linearity in terms of the sequence $\{f_k\}_{k \in {\mathbb N}^\star}$ of the Vlasov hierarchy, we get that each $F_k= \bar f_k-f^{\otimes k}$ is also a renormalized solution to the Vlasov Hierarchy~\eqref{vlasovhierarchy} for every $k$. Moreover by the identification above of the initial with zero initial data.

Furthermore, by~\eqref{estbarfk} and the assumption of Gaussian decay on~$f^0$, we have that 
\begin{equation}
\sup_{t\leq T^*}\int_{\Pi^{kd}\times\R^{kd}} |F_k|^q\,e^{\tilde\beta\,\sum_{i\leq k} (1+|v_i|)^2}\leq \tilde M^k,
\label{tildeM}
\end{equation}
for some~$\tilde\beta$ and some $\tilde M$. Eq.~\eqref{vlasovhierarchy} corresponds to Eq.~\eqref{hierarchyext} in the case $\gamma=0$, where $e_{k,\gamma}$ reduces to $e_{k,0}=\sum_{i\leq k} (1+|v_i|)^2$.
Hence, provided we choose some $\tilde \Lambda$, possibly lower than $\Lambda$ we satisfy all assumptions from Lemma~\ref{technicalLemma}.

Denoting by $Y_k = \int |F_k|^qe^{\tilde \lambda(t) e_{k,0}}$, we get for all $k\in {\mathbb N}^\star$
\[
Y_k(t) \le k\,\tilde L\, \int_0^t Y_{k+1} \, ds.
\]
We can then use Lemma \ref{indu} with $F_0=0$ up to any arbitrary $m$ to show, together with~\eqref{tildeM}, that
\begin{equation} \label{epsilon}
\begin{split}
  Y_k(t) &\le \tilde L^{m+1-k}\,\tilde M^{m+1}\,\int_0^t (t-s)^{m-k}\,\frac{m!}{(k-1)!\,(m-k)!}\,ds\\
  &\leq \tilde L^{m+1-k}\,\tilde M^{m+1}\,t^{m+1-k}\,\binom{m}{k-1}\leq 2^k\,\tilde M^k\,(2\,\tilde L\,\tilde M\,t)^{m+1-k}.
\end{split}
\end{equation}
By taking $t<T_0$ with $T_0$ small enough and sending $m$ to $\infty$, we obtain that $Y_k(t)=0$ and hence $\bar f_k=f^{\otimes k}$ on $[0,\ T_0]$. This allows to repeat the argument starting from~$t=T_0$ instead of $t=0$ until we reach the maximum time~$T^*$. This finally allows to conclude as claimed that $\bar f_k=f^{\otimes k}$ over the whole interval $[0,\ T^*]$. 

Coming back to our extracted subsequence on $f_{k,N}$, since all such subsequences have the same limit, we have convergence of the whole sequence to the $f^{\otimes k}$ concluding the proof.

\subsection{Proof of Theorem \ref{quantitative}}
The aim of this result is to provide a quantitative estimate between $f_{k,N}$ and $f_k$ that satisfy \eqref{hierarchy} and \eqref{vlasovhierarchy}, respectively, for the tensorized limits $f_k= 
{ f}^{\otimes k}$.   First let us note that $F_k^N= f_{k,N} -f_k$ satisfies,
\begin{equation}
\begin{split}
\partial_t F_k^N 
& + L_k F_k^N  \\
& + \frac{N-k}{N} \sum_{i=1}^k \nabla_{v_i} \cdot \int_{\Pi^d \times \R^d} F_{k+1}^N K(x_i - x_{k+1}) \, d x_{k+1} d v_{k+1}\\
&  = 
  \frac{\sigma^2}{2} \sum_{i=1}^k \Delta_{v_i} F_{k,N} + R_{k,N}
  \end{split} 
\end{equation}
where
$L_k $ is defined in \eqref{L}, 
and 
\begin{equation}
\begin{split}
R_{k,N}  =  &
 \sum_{i=1}^k \left[\Big(K\star \int_{\R^d} f \Big)(t,x_i) - \frac{1}{N} \sum_{j=1}^kK(x_i-x_j)\right] \cdot \nabla_{v_i} f_k. \\
& - \frac{N-k}{N} \sum_{i=1}^k \nabla_{v_i} \cdot \int_{\Pi^d\times \R^d} f_{k+1} K(x_i-x_{k+1}) d x_{k+1} d v_{k+1}.
\end{split}
\end{equation}
We again use Lemma~\ref{technicalLemma} with $q=2$ to deduce
\begin{equation} \label{FkN}
\begin{split}
  \frac{d}{dt} \int_{\Pi^{kd}\times \R^{kd}}& |F_{k,N}|^2 e^{\lambda(t) e_{k,\gamma}}
      +  \frac{\sigma^2}{4}\sum_{i\le k}  \int_{\Pi^{kd}\times \R^{kd}} |\nabla_{v_i} F_{k,N}|^2 e^{\lambda(t) e_{k}} \\
 &   \le k \frac{N- k}{N} \frac{C_{2,\sigma,d}}{\lambda^{\theta_{2,d}}(t)} \|K\|_{L^2}^2
      \int_{\Pi^{kd}\times \R^{kd}} |F_{k+1,N}|^2 e^{\lambda(t) e_{k+1}}\\
 &\quad +  \lambda'(t) \int_{\Pi^{kd}\times \R^{kd}}  e_{k} \, |F_{k,N}|^2 \, e^{\lambda(t) e_{k}}
       \\
 &\quad + \int_{\Pi^{kd}\times \R^{kd}} R_{k,N} F_{k,N}  e^{\lambda(t) e_{k}}. 
\end{split}
\end{equation}
Note that $R_{k,N}$ may be written as follows 
\begin{equation}
\begin{split}
R_{k,N}  =  &
 \sum_{i=1}^k \frac{1}{N} \sum_{j=1}^k \Bigg[(K\star \int_{\R^d} f)(t,x_i) - K(x_i-x_j)]\Bigg] \cdot \nabla_{v_i} f_k\\
& - \frac{N-k}{N} \sum_{i=1}^k 
\Bigg[ \nabla_{v_i} \cdot \int_{\Pi^d\times \R^d} f_{k+1} K(x_i-x_{k+1}) d x_{k+1} d v_{k+1}
         \\
&  \hspace{2.6cm} - \left(K \star \int_{\R^d} \overline f \right)(t,x_i)\cdot \nabla_{v_i} f_k\Bigg]. 
         \end{split}
\end{equation}
Then, using that $f_k= f^{\otimes k}$, we have
\[
\begin{split}
 \int_{\Pi^{kd}\times \R^{kd}} R_{k,N} F_{k,N}  e^{\lambda(t) e_{k}} 
  = 
\int_{\Pi^{kd}\times \R^{kd}} & \frac k N \sum_{i=1}^k  \Bigg[ \Big( K\star  \int_{\R^d} f  \Big)(t,x_i) \\&- K(x_i-x_1)\Bigg] \cdot \nabla_{v_i}  f_k \, F_{k,N}  e^{\lambda(t) e_{k}} ,
\end{split}
\]
where we have used the fact that the particles are interchangeable. Integrating by parts with respect to $v_i$ and using Young inequality, we obtain
\begin{equation} \label{Rest}
\begin{split}
 \int_{\Pi^{kd}\times \R^{kd}} R_{k,N} F_{k,N} e^{\lambda(t) e_{k}} \le &
     \frac{\sigma^2}{4} \frac k N \sum_{i=1}^k \int_{\Pi^{kd}\times \R^{kd}} |\nabla_{v_i} F_{k,N}|^2  e^{\lambda(t) e_{k}} 
   \\ &+\frac{1}{\sigma^2 }  \frac k N \sum_{i=1}^k \int_{\Pi^{kd}\times \R^{kd}} |\widetilde R_{k,N}^1|^2 e^{\lambda(t) e_{k}} 
\\
& +  \lambda(t)  \int_{\Pi^{kd}\times \R^{kd}} e_{k} \, |F_{k,N}|^2  \, e^{\lambda(t) e_{k}}  
\\ & +\frac{1}{2} \int_{\Pi^{kd}\times \R^{kd}} |\widetilde R_{k,N}^2|^2 \, e^{\lambda(t) e_{k}},
\end{split}
\end{equation}
where
\[
\begin{split}
\widetilde R_{k,N}^1 
   &= \Bigg[ \Big(K\star \int_{\R^d} f\, dx \Big)(t,x_i) - K(x_i-x_1)\Bigg] \,    f_k \,, \\
\widetilde R_{k,N}^2 
   &=  \sum_{i=1}^k \Bigg[ \Big(K\star \int_{\R^d} f\, dx \Big)(t,x_i) - K(x_i-x_1)\Bigg] \,    f_k \,.
\end{split}
\]
 We observe that 
\[\|\widetilde R_{k,N}^i\|^2_{L^2_{\lambda(t) e_k}}
   \le 
   C \, k \int_{\Pi^{kd}\times \R^{kd}} |f_{k}|^p e^{\lambda(t) e_{k}} \, ,
   \]
   with a constant $C$ that does not depend on $k$. We have also used the fact that, in particular, $K \in L^2(\Pi^d)$ and $ f \in L^\infty(\Pi^d\times \R^d)$.

   Then, using \eqref{Estim} and letting $N\to +\infty$, we get
$$
\sup_{t\le T^*} \int_{\Pi^{kd}\times \R^{kd}}|f_k|^p e^{\lambda(t)e_{k, \gamma}} \le 2^k F_0^k .
$$
We can insert this estimate into \eqref{Rest}, for $p= 2$, to derive
\[
\begin{split}
 \int_{\Pi^{kd}\times \R^{kd}} R_{k,N} F_{k,N} e^{\lambda(t) e_{k}} \le &
     \frac{\sigma^2}{4} \frac k N \sum_{i=1}^k \int_{\Pi^{kd}\times \R^{kd}} |\nabla_{v_i} F_{k,N}|^2  e^{\lambda(t) e_{k}} 
\\
& + \lambda(t)  \int_{\Pi^{kd}\times \R^{kd}}  e_{k} \, |F_{k,N}|^2  \, e^{\lambda(t) e_{k}}  
\\ & + C\, k \, 2^k F_0^k .
\end{split}
\]
Once this estimate is incorporated into \eqref{FkN} and using that $\lambda'(t) = - \frac{\lambda(t)}{1+t}$, we can, following the same lines of the proof of Proposition \ref{propLp}, repeat the estimate on the ODE inequality with the extra term coming from the interaction of $F_{k,N}$ with rest term $R_{k,N}$. This provides the conclusion that there exists $T^*$ such that
\[
\begin{split}
\sup_{t \le T^\star} \int_{\Pi^{kd}\times \R^{kd}} |f_{N,k} - f_k|^2 & e^{\lambda(t) e_{k, \gamma}} 
\\ & \le \widetilde C^k \varepsilon_N + \widetilde C^k \int_{\Pi^{kd}\times \R^{kd}} |f_{N,k}^0 - f_k^0|^2 e^{\lambda(0) e_{k, \gamma}}    ,
\end{split}
\]
where $\widetilde C$ is a positive constant that does not depend on $N$, and $\varepsilon_N = O(\varepsilon^N)$, where $\varepsilon < 1$ depends on $T^*$ that is small enough. This expression can be deduced  in a similar way as \eqref{epsilon} in the proof of Theorem \ref{Conv}.
  We  finally emphasize that the quantitative bounds of Theorem \ref{quantitative} would allow to recover the optimal convergence rate in $O(1/N)$ that was recently obtained in~\cite{Lacker}. 

\bigskip

\noindent {\bf Acknowledgments.} D. Bresch is partially supported by SingFlows project, grant ANR-18-CE40-0027. P.~E. Jabin is partially supported by NSF DMS Grant 161453, 1908739, 2049020. J. Soler is partially supported by the State Research Agency (SRA) of the Spanish Ministry of Science and Innovation and European Regional Development Fund (ERDF), project PID2022-137228OB-I00, and by Modeling Nature Research Unit, project QUAL21-011.

\end{document}